\documentclass[
	11pt,oneside,british,a4paper]{amsart}
\usepackage[T1]{fontenc}
\usepackage{enumitem}
\usepackage[utf8]{inputenc}
\usepackage[dvipsnames]{xcolor}
\usepackage[numbers]{natbib}
\usepackage{amstext,amsthm,amssymb}

\usepackage{amsrefs}
\usepackage{enumitem}
\usepackage{dsfont}
\usepackage{xparse,mathrsfs,mathtools}
\usepackage{bookmark}
\usepackage[capitalise,nameinlink]{cleveref}
\usepackage{babel}
\usepackage[yyyymmdd,hhmmss]{datetime}
\usepackage{hyperref}
\usepackage{lmodern,microtype}
\usepackage[a4paper, total={6in, 8.5in}]{geometry}
\hypersetup{
	unicode=true,%
	pdfdisplaydoctitle=true,%
	pdfpagemode=UseOutlines,%
	breaklinks=true,%
	pdfencoding=unicode,psdextra,
	pdfcreator={},
	pdfborder={0 0 0},
	hidelinks
}
\newtheorem{thm}{Theorem}[section]
\newtheorem{prop}[thm]{Proposition}
\newtheorem{lem}[thm]{Lemma}

\newcommand{\NN}{\mathbb{N}}
\newcommand{\ZZ}{\mathbb{Z}}

\newcommand{\RR}{\mathbb{R}}

\newcommand{\x}{\Tilde{x}}
\newcommand{\klj}{K^{l,j}}
\newcommand{\tlj}{H^{l,j}}

\newcommand{\clj}{\mathds{1}^{l,j}}
\newcommand{\maj}{K^{l,j}*\mathcal{F}^{-1}(\Psi_j)}
\newcommand{\slog}{S_n^l}
\newcommand{\tlog}{T_n^l}
\newcommand{\mq}[1]{M[{#1}^{q}]}

\DeclarePairedDelimiter{\abs}{\lvert}{\rvert}
\DeclarePairedDelimiter{\norm}{\lVert}{\rVert}

\DeclareRobustCommand{\[}{\begin{equation*}}
\DeclareRobustCommand{\]}{\end{equation*}}

\def\XXint#1#2#3{{\setbox0=\hbox{$#1{#2#3}{\int}$ }
\vcenter{\hbox{$#2#3$ }}\kern-.6\wd0}}

\title[A Cotlar Type Maximal Function Associated With Fourier Multipliers]
{A Cotlar Type Maximal Function Associated With Fourier Multipliers}
\date{\today}


\author{Rajula~Srivastava}
\address{
Department of Mathematics\\ 
University of Wisconsin-Madison\\
480 Lincoln Dr, Madison\\
WI-53706\\
USA}
\email{rsrivastava9@wisc.edu}

\begin{document}
\begin{abstract}
We prove the $L^p$ boundedness of a maximal operator associated with a dyadic frequency decomposition of a Fourier multiplier, under a weak regularity assumption.
\end{abstract}
\maketitle

\section{Introduction}
Consider a Mikhlin-Hörmander multiplier $m$ on $\RR^d$ satisfying the assumption
\begin{equation}
\abs{\partial^\gamma m(\xi)}\leq A\abs{\xi}^{-\abs{\gamma}}
\label{mhtype}
\end{equation}
for all multi-indices $\gamma$ with $\abs{\gamma}\leq L$ for some integer $L>d$.

Let $\chi \in C_c^{\infty}(\RR)$ be supported in $(1/2,2)$ such that $\sum_{j=-\infty}^{\infty}\chi(2^jt)=1$ and let $\phi=\chi(\lvert.\rvert)$. For a given Schwartz function $f$, let $Sf:=\mathcal{F}^{-1}[m\hat{f}]$, and for $n\in \ZZ$ let $S_n$ be defined by
\begin{equation}
\widehat{S_nf}(\xi):=\sum_{j\leq n}\phi(2^{-j}\xi)m(\xi)\hat{f}(\xi).
\label{opdefn}
\end{equation}
We are interested in bounds for the maximal function 
\begin{equation}
S_*f(x):=\underset{n\in \ZZ}{\mathrm{sup}}\,\abs{S_nf(x)}.
\label{opdefnmax}
\end{equation}
The above operator was studied by Guo, Roos, Seeger and Yung in [\cite{guo2019maximal}, in connection with proving $L^p$ bounds for a maximal operator associated with families of Hilbert transforms along parabolas. The multiplier $m$ in [\cite{guo2019maximal} was assumed to satisfy the condition
\begin{equation}
\label{initialassump}
\sup_{t>0}\,\norm{\phi m(t\cdot)}_{\mathcal{L}^1_{\beta}}=B(m)<\infty   
\end{equation}
with $\beta>d$. Here $\mathcal{L}^1_{\beta}$ is the potential space of functions $g$ with $(I-\Delta)^{\beta/2}g\in L^1$ (we note the analogy with condition (\ref{mhtype}) here). With the above hypothesis, the authors were able to prove a pointwise Cotlar-type inequality
\begin{equation}
\label{initialCotlar}
S_*f(x)\leq \frac{1}{(1-\delta)^{1/r}}(M(\abs{Sf}^r)(x))^{1/r}+C_{d,\beta}\delta^{-1}B(m)Mf(x)    
\end{equation}
for $f\in L^p(\RR^d)$ and for almost every $x$ (with $ r>0$ and $0<\delta\leq 1/2)$. Here $M[f]$ denotes the standard Hardy-Littlewood Maximal function. (\ref{initialCotlar}) easily implies $L^p$ and weak (1,1) bounds for the maximal operator $S_*$. 

It is natural to ask if one could weaken the assumption (\ref{initialassump}) and still establish $L^p$ bounds on the operator $S_*$, possibly without the intermediate step of proving a pointwise inequality of the form (\ref{initialCotlar}) for all exponents $r$. In this paper, we answer this question in the affirmative. In particular, we show that for the $L^p$ bounds on $S_*$ to hold, it is enough for the multiplier $m$ to satisfy a much weaker condition 
\begin{equation}
\sup_ {t>0} \int | \mathcal{F}^{-1} [\phi m(t\cdot)](x)| (\log{(2+|x|)})^{\alpha} dx :=B(m)<\infty,\qquad \alpha >3.    
\label{assump}
\end{equation}
Our main result is the following:
\begin{thm}\label{main thm}
$S_*$, as defined in (\ref{opdefnmax}) for a multiplier $m$ satisfying (\ref{assump}), is of weak-type (1,1) and bounded on $L^p$ for $p\in (1,\infty)$, with the respective operator norm $\lesssim_p B(m)$.
\end{thm}
We remark here that it is not possible to do away with the smoothness assumption entirely. In other words, the condition
\begin{equation*}
\sup_ {t>0} \int | \mathcal{F}^{-1} [\phi m(t\cdot)](x)| dx <\infty   \end{equation*}
alone is not enough to guarantee $L^p$ bounds on the maximal operator $S_*$, or even on the singular operator $S$. For counterexamples, we refer to [\cite{littman1968LpMultiplier}, Section 5 and [\cite{stein1967boundedness}, Section 3.

We shall denote the Littlewood-Paley pieces of $m$ by $m_j$. More precisely, for $j\in \ZZ$, we define $m_j(\xi):=\phi(2^{-j}\xi)m(\xi)$. Furthermore let $a_j(\xi)=m_j(2^j\xi)=\eta(\xi)m(2^j\xi)$. Observe that $supp(m_j)\subset (2^{j-1},2^j)$ and $supp(a_j)\subset (1/2,2)$. Let $K_j=\mathcal{F}^{-1}[m_j]$. Then (\ref{opdefn}) can be re-written as
\begin{equation}
    S_nf(x)=\sum_{j\leq n} K_j*f(x).
    \label{opdefnspace}
\end{equation}

However, in order to quantify the smoothness condition in (\ref{assump}), it is useful to partition $K_j$ (for each $j\in \ZZ$) on the space side as well (see [\cite{grafakos2006maximal}). To this effect, let $\eta_0\in C_c^\infty(\RR^d)$ be such that $\eta_0$ is even, $\eta_0(x)=1$ for $|x|\leq 1/2$ and $\eta_0$ is supported where $|x|\leq 1$. For $l\in \NN$, let $\eta_l(x)=\eta_0(2^{-l}x)-\eta_0(2^{-l+1}x)$. For $l\in \NN\cup \{0\}$ we define
\begin{equation}
K^{l,j}(x)=\eta_l(2^jx)\mathcal{F}^{-1}[m_j](x).
\label{kljdefn}
\end{equation}
By the assumption (\ref{assump}), we then have
\begin{equation}
\label{kljnorm}
\norm{K^{l,j}}_{L^1}\lesssim B(\log{(2+2^l)})^{-\alpha}.
\end{equation}

Now the multiplier corresponding to $K^{l,j}$ is given by $2^{-jd}\hat{\eta_l}(2^{-j}\cdot)*m_j$, which, unlike $m_j$, is not compactly supported. However the rapid decay of $\hat{\eta_l}$ still leads to the multiplier "essentially" being supported in a slightly thicker (but still compact) dyadic annulus, with the other frequency regions contributing negligible error terms. We make these ideas rigorous in Section \ref{error}. The arguments used are similar in spirit to those in [\cite{carbery1986variants}, Section 5. Another source of reference is [\cite{seeger1988some}.
As a corollary, we prove that the singular operator
\begin{equation}
    S^lf(x)=\sum_{j\in \ZZ}\klj*f(x)
    \label{op_l}
\end{equation}
is bounded on $L^2$, with operator norm $\lesssim Bl^{-\alpha}$.

In Section \ref{majsec}, we use Bernstein's inequality (see [\cite{wolff2003lectures}) to establish $L^p$ bounds for the aforementioned portion with the major contribution and with compact frequency support. We also establish a pointwise estimate on its gradient.

In Section \ref{weakSl}, using the estimates from Section \ref{majsec}, we prove that the Calder\'on-Zygmund operator $S^l$ associated to the multiplier $m$ is of weak type (1,1), with the operator norm $\lesssim Bl^{-\alpha+1}$.  We do so by establishing the result for the operator $T^lf:=\sum_{j\in \ZZ}\tlj*f$, where $\tlj$ is the portion of the kernel $\klj$ with the major contribution. The arguments flow in the same vein as those in the proof of the Mikhlin-Hörmander Multiplier Theorem (see [\cite{grafakos2008classical},[\cite{H}). We also establish $L^p$ bounds on $T^l$ using interpolation.

In Section \ref{largesec}, we investigate the properties of the truncated operator $\tlog f:= \sum_{j\leq n} \tlj*f$ (for $n\in \ZZ$ and $l\geq 0$), with an aim of establishing $L^p$ bounds for the associated maximal operator $T^l_*f=\text{sup }\,|\tlog f|$. 
 
We wish to show that for $l\geq 0$, the operator $T^l_*$ is bounded on $L^p$ for $1<p<\infty$ with the corresponding operator norm $\lesssim_p Bl^{-\alpha+1}$ ($\lesssim B$ for $l=0$). Since our assumption (\ref{assump}) on $m$ is much weaker than that in [\cite{guo2019maximal}, a pointwise inequality like (\ref{initialCotlar}) for all $r>0$ seems out of reach. However, by using similar ideas, we are able to establish a Cotlar type inequality
\[
T^l_* f(x)\lesssim_d \frac{1}{(1-\delta)^{1/p}}(M(\abs{T^lf}^p)(x))^{1/p}+2^\frac{ld}{p}l^{-\alpha+1}(1+\delta^{-1/p})B(m)(M(\abs{f}^p)(x))^{1/p}.
\label{cotlar}
\]
for each $T^l_*$ and for a large enough exponent $p$ (here $0<\delta\leq 1/2$). Roughly speaking, a choice of $p=p_l \sim l$ will work. In particular, $p_l\rightarrow \infty$ as $l\rightarrow\infty$. Using this inequality, we can conclude that $T^l_*$ is bounded (with norm $\lesssim_pBl^{-\alpha+1})$ for all $p\in [p_l,\infty)$. This idea of keeping track of the explicit dependence on the exponent $l$ and using the decay in $l$ to sum the pieces up has also been used in [\cite{grafakos2006maximal}, albeit for a different maximal operator than the one considered here. For $p\in (1, p_l)$, however, we rely on a weak (1,1) estimate for $T^l_*$ (which is not hard to obtain) and then an interpolation between $1$ and $p_l$, which causes us to gain a power of $l$. In other words, we are only able to retain a decay of $l^{-\alpha+2}$ (hence the assumption $\alpha>3$). This is a trade off of working with  a weaker logarithmic regularity assumption. 

Finally, in Section \ref{weakmax}, we establish a weak (1,1) estimate for the maximal operator $T^l_*$ and obtain $L^p$ bounds for the sum $\sum_{l\geq 0}T^l_*$, and consequently for $S_*$ (we use the decay in $l$ and the condition $\alpha>3$ here).

\subsection*{Acknowledgements} The author would like to thank her advisor Andreas Seeger for introducing this problem, for his guidance and several illuminating discussions. Research supported in part by NSF grant 1500162.

\section{The Error Terms}
\label{error}
Let $\Psi\in C^\infty$ with $\Psi=1$ on $\{1/4\leq |\xi|\leq 4\}$ and supported on $\{1/5\leq |\xi|\leq 5\}$. For $j\in \ZZ$, define $\Psi_j(\xi)=\Psi(2^{-j}\xi)$. Then
\begin{equation}
\sum_{j\in \ZZ}\klj= \sum_{j\in \ZZ}\klj*\mathcal{F}^{-1}(\Psi_j) + \sum_{j\in \ZZ}\klj *\mathcal{F}^{-1}(1-\Psi_j)
\label{twoterms}
\end{equation}
where $\klj$ is as defined in (\ref{kljdefn}).

In this section, we will show that the contribution from the second sum above can be made as small as required using the rapid decay of $\hat{\eta}$. To control this sum, we study the corresponding multiplier given by
\[
\sum_{j\in \ZZ}2^{-jd}(1-\Psi_j)(\xi)\hat{\eta_l}(2^{-j}\cdot)*m_j(\xi)=\sum_{j\in \ZZ} (1-\Psi_j(\xi))2^{-jd}\int m_j(\omega)\hat{\eta_l}(2^{-j}(\xi-\omega))\, d\omega .
\]
\begin{lem}
\label{errorterms}
Let $l\geq 0$. For any $N\in \NN$, multi-index $\gamma$ and $\xi\neq 0$, we have the estimate
\begin{equation}
\sum_{j\in \ZZ}2^{-jd}\abs{(1-\Psi_j)\partial^{\gamma}_\xi (\hat{\eta_l}(2^{-j}\cdot)*m_j)(\xi)} \lesssim_{\eta, N, \gamma,d} B2^{l(d+\abs{\gamma}-N)} \abs{\xi}^{-\abs{\gamma}}.  
\end{equation}
\end{lem}

\begin{proof}
Fix $\xi\neq 0$. Let $k\in \ZZ$ be such that $2^{k-1}\leq \abs{\xi}<2^{k+1}$.As $\Psi_k=1$ on $\{2^{k-2}\leq \abs{\xi}\leq 2^{k+2}\}$, we can split the sum under consideration into two parts
\begin{align*}
&\sum_{j\in \ZZ}2^{-jd}\abs{(1-\Psi_j)\partial^{\gamma}_\xi (\hat{\eta_l}(2^{-j}\cdot)*m_j)(\xi)} \leq \\
&\sum_{j> k+2}2^{-jd}\abs{\partial^{\gamma}_\xi (\hat{\eta_l}(2^{-j}\cdot)*m_j)(\xi)}
+\sum_{j<k-2}2^{-jd}\abs{\partial^{\gamma}_\xi (\hat{\eta_l}(2^{-j}\cdot)*m_j)(\xi)}.
\end{align*}

\subsection*{First Term:} We have
\begin{align*}
&\sum_{j>k+2}2^{-jd}\abs{\partial^{\gamma}_\xi (\hat{\eta_l}(2^{-j}\cdot)*m_j)(\xi)} 
\leq &\sum_{j>k+2}2^{-jd}\int_{2^{j-1}\leq \abs{\omega}\leq 2^{j+1}}\abs{ m_j(\omega)}\abs{\partial^\gamma_\xi\hat{\eta_l}(2^{-j}(\xi-\omega))}\, d\omega.\\
\end{align*}
Now we observe that for $2^{j-1}\leq \abs{\omega}\leq 2^{j+1}$, we have $\abs{\omega-\xi}\geq \abs{\omega}/2$ and $2^{-j}\abs{\omega}/2\sim 1$. Hence, the Schwartz decay of $\hat{\eta}$ yields 
\begin{align*}
&\sum_{j>k+2}2^{-jd}\abs{\partial^{\gamma}_\xi (\hat{\eta_l}(2^{-j}\cdot)*m_j)(\xi)} \\
&\lesssim_{\eta,N,\gamma}\sum_{j>k+2} 2^{-jd}\int_{2^{j-1}\leq \abs{\omega}\leq 2^{j+1}}\norm{ m}_{L^\infty}2^{l(d+\abs{\gamma})}2^{-j\abs{\gamma}}\Big(2^{l-j}\abs{\omega}/2\Big)^{-N}\, d\omega\\
&\lesssim_{d,\eta,N,\gamma}\sum_{j>k+2}2^{-jd}\norm{m}_{L^\infty}2^{l(d+\abs{\gamma}-N)}2^{-j\abs{\gamma}}\int_{2^{j-1}\leq \abs{\omega}\leq 2^{j+1}}\, d\omega\\
&\lesssim_{d,\eta,N,\gamma}\sum_{j>k+2}2^{d}\norm{m}_{L^\infty}2^{l(d+\abs{\gamma}-N)}2^{-j\abs{\gamma}}
\lesssim_{d,\eta,N,\gamma}B2^{l(d+\abs{\gamma}-N)}2^{-(k+2)\abs{\gamma}}
\lesssim_{d,\eta,N,\gamma}B2^{l(d+\abs{\gamma}-N)}\abs{\xi}^{-\abs{\gamma}}.
\end{align*}
\subsection*{Second Term:} In this case for $2^{j-1}\leq \abs{\omega}\leq 2^{j+1}$ we have that $\abs{\omega-\xi}\geq \abs{\xi}/2$ and $2^{-j}\abs{\xi}/2\geq 1$. Hence 
\begin{align*}
&\sum_{j<k-2}2^{-jd}\abs{\partial^{\gamma}_\xi (\hat{\eta_l}(2^{-j}\cdot)*m_j)(\xi)}\\ 
&\lesssim_{\eta,N,\gamma}\sum_{j<k-2}2^{-jd}\int_{2^{j-1}\leq \abs{\omega}\leq 2^{j+1}}\norm{ m}_{L^\infty}2^{l(d+\abs{\gamma})}2^{-j\abs{\gamma}}\Big(2^{l-j}\abs{\xi}/2\Big)^{-N}\, d\omega\\
&\lesssim_{\eta,N,\gamma}\sum_{j<k-2}2^{-jd}\norm{m}_{L^\infty}2^{l(d+\abs{\gamma}-N)}\int_{2^{j-1}\leq \abs{\omega}\leq 2^{j+1}}\abs{\xi}^{-\abs{\gamma}}\Big(2^{-j}\abs{\xi}/2\Big)^{-N+\abs{\gamma}}\, d\omega\\
&\lesssim_{\eta,N,\gamma,d}\norm{m}_{L^\infty}2^{l(d+\abs{\gamma}-N)}\abs{\xi}^{-\abs{\gamma}}\sum_{j<k-2}2^{-jd}2^{(k-1)d}
\lesssim_{\eta,N,\gamma,d}\norm{m}_{L^\infty}2^{l(d+\abs{\gamma}-N)}\abs{\xi}^{-\abs{\gamma}}.
\end{align*}
\end{proof}

Now a simple application of the Mikhlin-H\"ormander Multiplier theorem gives us
\begin{thm}
\label{errorthm}
Let $l\geq 0$. For any $N_0\in \NN$, $p\in (1,\infty)$ and Schwartz function $f$, we have
\begin{equation*}
    \norm{\sum_{j\in \ZZ}\klj *\mathcal{F}^{-1}(1-\Psi_j)*f}_{L^p}\lesssim_d \text{max }(p,(p-1)^{-1})C_{\eta,N_0,p,\Psi}B2^{-log_0}\norm{f}_{L^p}.
\end{equation*}
Furthermore, we also have \[\norm{\sum_{j\in \ZZ}\klj *\mathcal{F}^{-1}(1-\Psi_j)*f\in L^{1,\infty}}_{L^{1,\infty}}\lesssim_d C_{\eta,N_0,p,\Psi}B2^{-log_0}\norm{f}_{L^1}.\]
\end{thm}
\begin{proof}
We need to prove that
\[\abs{\sum_{\beta\leq \gamma}c_{\beta,\gamma}\sum_{j\in \ZZ}\partial^\beta(1-\Psi_j)(\xi)\partial^{\gamma-\beta}(2^{-jd}\hat{\eta_l}(2^{-j}\cdot)*m_j)(\xi)}\lesssim_d C_{\eta,N,p,\Psi}B2^{-log_0}\abs{\xi}^{-\abs{\gamma}}\]
where $\gamma$ is a multi-index with $\abs{\gamma}\leq d/2+1$ and $\xi\neq 0$. We observe that for all values of $j$ except for $j_1,j_2$ where $2^{j_1+2}\leq \abs{\xi}<5\cdot2^{j_1}$ or when $2^{j_2}/5\leq \abs{\xi}<2^{j_2-2}$, we can use lemma \ref{errorterms} directly (as then $\abs{(1-\Psi_j)(\xi)}$ is a constant). Further, for the two remaining cases, we observe that $\abs{\xi}\sim 2^{j_k}$ $(k=1,2)$. Hence we can bound the term $\abs{\partial^\beta(1-\Psi_{j_k})}$ by $C_{\Psi}\abs{\xi}^{-\abs{\beta}}$ and apply the previous lemma to the term $\abs{\partial^{\gamma-\beta}(2^{-jd}\hat{\eta_l}(2^{-j}\cdot)*m_j)(\xi)}$ to bound it above by $C_{\eta,N}B2^{l(d+\abs{\gamma}-\abs{\beta}-N_0-\abs{\gamma})}\abs{\xi}^{-(\abs{\gamma}-\abs{\beta})}$. The result then follows by summing up.
\end{proof}

As a consequence, we obtain the $L^2$ boundedness of $S^l$ (as defined in (\ref{op_l})), with polynomial decay in $l$. 
\begin{thm}
\label{l2bound}
For $f\in L^2$ and $l>0$, we have
\[\norm{S^l(f)}_{L^2}\lesssim Bl^{-\alpha}\norm{f}_{L^2}.\]
We also have 
\[\norm{S^0(f)}_{L^2}\lesssim B\norm{f}_{L^2}.\]
\end{thm}
\begin{proof}
It is enough to prove the above for a Schwartz function $f$. Now
\[S^l(f)=\sum_{j\in \ZZ}\klj*f=\sum_{j\in \ZZ}\klj*\mathcal{F}^{-1}(\Psi_j)*f + \sum_{j\in \ZZ}\klj *\mathcal{F}^{-1}(1-\Psi_j)*f.\]
As we have already established the $L_2$ boundedness of the second term in Theorem \ref{errorthm} (with as good a decay in $l$ as required), we only need to prove the theorem for the first term. To this effect, let 
\[f=\sum_{k\in \ZZ}\Delta_k f\]
be a Littlewood-Paley decomposition of $f$. Then
\begin{align*}
\norm{\sum_{j\in \ZZ}\klj*\mathcal{F}^{-1}(\Psi_j)*f}_{L^2}
&\lesssim \norm{\sum_{j\in \ZZ}\sum_{k\in \ZZ}\klj*\mathcal{F}^{-1}(\Psi_j)*\Delta_k f}_{L^2}\\
&\lesssim \norm{\sum_{j\in \ZZ}\klj*\mathcal{F}^{-1}(\Psi_j)*(\Delta_{j-1}+\Delta_{j}+\Delta_{j+1}) f}_{L^2}
\end{align*}
where we have used the fact that the frequency support of $\klj*\mathcal{F}^{-1}(\Psi_j)$ is contained in $\{2^j/5\leq \abs{\xi}\leq 5.2^j\}$. From (\ref{kljnorm}), we also have that 
\[\norm{\klj*\mathcal{F}^{-1}(\Psi_j)}_{L^1}\leq \norm{\klj}_{L^1}\norm{\mathcal{F}^{-1}(\Psi_j)}_{L^1}\lesssim B (\log{(2+2^l)})^{-\alpha}\norm{\Psi}_{L^\infty}. \]
Hence
\begin{align*}
\norm{\sum_{j\in \ZZ}\klj*\mathcal{F}^{-1}(\Psi_j)*f}_{L^2}
&\lesssim B (\log{(2+2^l)})^{-\alpha} \norm{\sum_{j\in \ZZ}(\Delta_{j-1}+\Delta_{j}+\Delta_{j+1}) f}_{L^2}\\
&\lesssim B (\log{(2+2^l)})^{-\alpha}\norm{f}_{L^2}
\end{align*}
which proves the theorem.
\end{proof}

\section{Estimates for the Majorly Contributing Portion of the Kernel}
\label{majsec}
We now turn our attention to the first term in (\ref{twoterms}), which is the one with the main contribution. The main advantage we have now is that the $jth$ term in
\[\sum_{j\in \ZZ}\klj*\mathcal{F}^{-1}(\Psi_j)\]
has frequency supported in the annulus $\{2^j/5\leq \abs{\xi}\leq 5\cdot2^j\}$. Hence, we can use Bernstein's inequality to get bounds on the $L_p$ norm (of each term, with $1<p<\infty$) and $L^{\infty}$ norm (of the derivative). 
\begin{prop}\label{new}
Let $q\in (1,\infty)$, and let $q'$ denote the Hölder conjugate exponent of $q$. Then for all $l\in \NN\cup \{0\}$ and $j\in \ZZ$, we have

\begin{enumerate}
    \item $\norm{\klj*\mathcal{F}^{-1}(\Psi_j)}_{L^{q'}}\lesssim B 2^{\frac{jd}{q}}(\log{(2+2^l)})^{-\alpha}$.
    \item For $x\in \RR^d$, $|\nabla (\klj*\mathcal{F}^{-1}(\Psi_j))(x)|\lesssim  2^{j}|\klj*\mathcal{F}^{-1}(\Psi_j)(x)|$.
\end{enumerate}

\end{prop}
\begin{proof}
For the first part, we have
\begin{align*}
&\norm{\klj*\mathcal{F}^{-1}(\Psi_j)}_{L^{q'}}\\
&\leq \norm{\klj}_{L^1}\norm{\mathcal{F}^{-1}(\Psi_j)}_{L^{q'}}
\lesssim B(\log{(2+2^l)})^{-\alpha}\norm{\Psi_j}_{L^q}
\lesssim_{\Psi} B(\log{(2+2^l)})^{-\alpha}2^{jd/q}
\end{align*}
where in the second step we have used (\ref{kljnorm}) and Hausdorff-Young's inequality.

For the second part, we recall that the Fourier transform of $\nabla (\klj*\mathcal{F}^{-1})(\Psi_j)$ is supported on a dyadic annulus of radius $2^j$. The assertion then follows from Bernstein's inequality (see [\cite{wolff2003lectures}, Proposition 5.3). 
\end{proof}

\section{Weak (1,1) Boundedness of \texorpdfstring{$S^l$}{}}
\label{weakSl}
Let $\mathds{1}_l$ be the characteristic function of the set $\{x: 2^{l-1}\leq \abs{x}\leq 2^{l+1}\}$ for $l>0$ and of the set $\{x: \abs{x}\leq 1\}$ for $l=0$. We will denote $\mathds{1}_l(2^jx)$ by $\clj(x)$. Then (\ref{twoterms}) can be rewritten as
\begin{equation}
\label{twotermswithchar}
\sum_{j\in \ZZ}\klj= \sum_{j\in \ZZ}\klj\clj=\sum_{j\in \ZZ}\klj*\mathcal{F}^{-1}(\Psi_j)\clj + \sum_{j\in \ZZ}\klj *\mathcal{F}^{-1}(1-\Psi_j)\clj.
\end{equation}
The advantage of (\ref{twotermswithchar}) over (\ref{twoterms}) is that it preserves information about the compact support of the kernel $\klj$, a property we will exploit quite often in the forthcoming proofs. Now for $f\in L^1(\RR^d)$, we can estimate
\[
\norm{\sum_{j\in \ZZ}\klj*f}_{L^{1,\infty}}
\leq  \norm{\sum_{j\in \ZZ}\tlj*f}_{L^{1,\infty}} + \norm{\sum_{j\in \ZZ}(\klj *\mathcal{F}^{-1}(1-\Psi_j)\clj)*f}_{L^{1,\infty}}
\]
where we define $\tlj:=(\maj)\clj$. Also let the operator $T^l$ be defined as $T^lf:= \sum_{j\in \ZZ}\tlj*f$.

By Theorem \ref{errorthm}, we conclude that $\norm{\sum_{j\in \ZZ}(\klj *\mathcal{F}^{-1}(1-\Psi_j)\clj)*f}_{L^{1,\infty}}\lesssim Bl^{-\alpha+1}\norm{f}_{L^1}$. 
Hence, in order to prove that $S^l$ is of weak type $(1,1)$ (with the respective norm $\lesssim_d Bl^{-\alpha+1}$), it is enough to prove the same for $T^l$, which is the content of the next theorem.  The proof we give here essentially uses the same ideas as Hörmander's original proof of the (Hörmander-)Mikhlin Multiplier Theorem (see [\cite{H}, also [\cite{grafakos2008classical}).

\begin{thm}\label{weak thm}
For all $f\in L^1(\RR^d)$, we have
\[\norm{T^lf}_{L^{1,\infty}}\lesssim Bl^{-\alpha+1}\norm{f}_{L^1}\]
for $l\geq 1$ and
\[\norm{T^0f}_{L^{1,\infty}}\lesssim B\norm{f}_{L^1}\]
\end{thm}
\begin{proof}
We prove the result for $l>0$. The result for $l=0$ follows similarly. Also, for this proof, we can assume that $B=1$. Let $f\in L^1(\RR^d)$ and fix $\sigma >0$. Let
\[
f= f_0+f_1
\]
be the standard Calder\'on-Zygmund decomposition of $f$ at the level $l^{\alpha-1}\sigma$. More precisely, let $\{I_k\}_{k\in \NN}$ be axis-parallel cubes with centres $\{a_k\}_{k\in \NN}$ respectively such that
\begin{align*}
&l^{\alpha-1}\sigma<\abs{I_k}^{-1}\int_{I_k}\abs{f(y)}\, dy\leq 2^dl^{\alpha-1}\sigma,\\
&\abs{f(x)}\leq l^{\alpha-1}\sigma \mathrm{ \,a.e.\, for\, } x\not\in \underset{k\in \NN}{\bigcup}I_k,
\end{align*}

\[
f_0(x)=
\begin{cases}
f(x)-\abs{I_k}^{-1}\int_{I_k}f(y)\, dy \: x\in I_k,\\
0,\,\textrm{otherwise.} \\
\end{cases}
f_1(x)=
\begin{cases}
\abs{I_k}^{-1}\int_{I_k}f(y)\, dy \: x\in I_k,\\
f(x),\,\textrm{otherwise.} 
\end{cases}
\]
Now
\begin{align*}
&\textrm{meas}(\{x\colon \abs{\sum_{j\in \ZZ}\tlj*f(x)}>\sigma \})\\ 
\leq &\textrm{meas}(\{x\colon \abs{\sum_{j\in \ZZ}\tlj*f_0(x)}>\sigma/2 \})+\textrm{meas}(\{x\colon \abs{\sum_{j\in \ZZ}\tlj*f_1(x)}>\sigma/2 \})
\end{align*}
We estimate $\textrm{meas}(\{x\colon \abs{\sum_{j\in \ZZ}\tlj*f_0(x)}>\sigma/2 \})$
\begin{align*}
&\leq \textrm{meas}(\{x\colon \abs{\sum_{j\in \ZZ}\tlj*f_0(x)}>\sigma/2 \}\bigcap (\underset{k\in \NN}{\bigcup}2I_k)^c)+ \textrm{meas}(\underset{k\in \NN}{\bigcup}2I_k)\\
&\lesssim_d \textrm{meas}(\{x\colon \abs{\sum_{j\in \ZZ}\tlj*f_0(x)}>\sigma/2 \}\bigcap (\underset{k\in \NN}{\bigcup}2I_k)^c)+\frac{l^{-\alpha+1}\norm{f}_1}{\sigma}.
\tag{a}
\label{a}
\end{align*}
Now since the mean value of $f_0$ over $I_k$ vanishes, we have
\begin{align*}
&\int_{(\underset{k\in \NN}{\bigcup}2I_k)^c}\abs{\sum_{j\in \ZZ}\tlj*f_0(x)}\, dx\\
&\leq  \sum_{k\in \NN}\int_{I_k}\Bigg(\int_{(\underset{k\in \NN}{\bigcup}2I_k)^c}\abs{\sum_{j\in \ZZ}\tlj(x-y)-\tlj(x-a_k)}\,dx\Bigg )\abs{f_0(y)}\, dy\\
&\lesssim_d l^{-\alpha+1}\sum_{k\in \NN}\int_{I_k}\abs{f_0(y)\, dy}
\lesssim_d l^{-\alpha+1}\norm{f}_1,
\tag{b}
\label{b}
\end{align*}
provided we prove that 
\[
\int_{(\underset{k\in \NN}{\bigcup}2I_k)^c}\abs{\sum_{j\in \ZZ }\tlj(x-y)-\tlj(x-a_k)}\,dx\lesssim_d l^{-\alpha+1},\: y\in I_k.
\tag{c}
\label{c}
\]
We postpone the proof of (\ref{c}) in order to conclude the estimates. By (\ref{a}), (\ref{b}) and (\ref{c}), we obtain
\[
\textrm{meas}(\{x\colon \abs{\sum_{j\in \ZZ}\tlj*f_0(x)}>\sigma/2 \})\lesssim_d \sigma^{-1}l^{-\alpha+1}\norm{f}_1 +\sigma^{-1}l^{-\alpha+1}\norm{f}_1\lesssim_d \sigma^{-1}l^{-\alpha+1}\norm{f}_1.
\tag{d}
\label{d}
\]
Set $p=2l+4$. Then by Theorem \ref{l2bound}, we have $\norm{S^lf}_2\lesssim_d l^{\alpha-1}\norm{f}_{L^2},$ and hence
\[
\sigma^2\,\textrm{meas}(\{x\colon \abs{\sum_{j\in \ZZ}\tlj*f_1(x)}>\sigma/2 \})
\leq \norm{\sum_{j\in \ZZ}\tlj*f_1}_2^2\lesssim l^{2(-\alpha+1)}\norm{f_1}_2^2
\]
\begin{align*}
&=l^{2(-\alpha+1)}\Bigg (\sum_{k\in \NN}\abs{I_k}^{-1}\Big\lvert\int_{I_k}f(x)\,dx\Big \rvert^2+\int_{(\underset{k\in \NN}{\bigcup}I_k)^c}\abs{f(x)}^2\, dx\Bigg ) \\
&\lesssim_d l^{-\alpha+1}\sigma\Big \{ \sum_{k\in \NN}\Big\lvert\int_{I_k}f(x)\,dx\Big \rvert+\int_{(\underset{k\in \NN}{\bigcup}I_k)^c}\abs{f(x)}\, dx\Big \}
\end{align*}
which enables us to conclude
\[
\textrm{meas}(\{x\colon \abs{\sum_{j\in \ZZ}\tlj*f_1(x)}>\sigma/2 \})\lesssim \frac{l^{-\alpha+1}\norm{f}_1}{\sigma}.
\tag{e}
\label{e}
\]
Combining the estimates (\ref{d}) and (\ref{e}) yields
\[
\textrm{meas}(\{x\colon \abs{\sum_{j\in \ZZ}\tlj*f(x)}>\sigma \})\lesssim_d \frac{l^{-\alpha+1}\norm{f}_1}{\sigma}.
\]
There remains the proof of (\ref{c}). It is sufficient to prove that
\[
\sum_{j\in \ZZ}\int_{\abs{x}\geq 2t}\abs{\tlj(x-y)-\tlj(x)}\, dx\lesssim l^{-\alpha+1}\; (\abs{y}\leq t, t>0).
\]
Fix $t>0$. Let $m\in \ZZ$ such that $\abs{t}\sim 2^m$.
Now for $j> -m$,
\begin{align*}
\int_{\abs{x}\geq 2t}\abs{\tlj(x-y)-\tlj(x)}\, dx
&\leq 2\int_{\abs{x}\geq 2t}\abs{\tlj(x)}\, dx
\leq\int_{\abs{x}\geq 2^m}\abs{\klj*\mathcal{F}^{-1}(\Psi_j)(x)}\clj(x)\, dx.
\end{align*}
Now we observe that for the last term to be non-zero, $2^{l-j}\geq 2^m$. Hence we have \begin{align*}
&\sum_{j>-m}\int_{\abs{x}\geq 2t}\abs{\tlj(x-y)-\tlj(x)}\, dx\\
&\leq \sum_{-m< j\leq l-m} \int_{\abs{x}\geq 2^m}\abs{\klj*\mathcal{F}^{-1}(\Psi_j)(x)}\clj(x)\, dx 
\leq l \norm{\klj}_{L^1}\norm{\mathcal{F}^{-1}(\Psi_j)}_{L^1}\lesssim_{\Psi} l(\log{(2+2^l)})^{-\alpha}\\
&\lesssim l^{-\alpha+1}
\end{align*}
where we have used our assumption (\ref{assump}) to conclude the second to last inequality.
Moreover, by Proposition \ref{new}, we have 
$
\abs{\nabla (\klj*\mathcal{F}^{-1}(\Psi_j))(x)}\lesssim 2^{j} \abs{ (\klj*\mathcal{F}^{-1}(\Psi_j))(x)}
$
which yields

\begin{align*}
\sum_{j\leq -m }\int_{\abs{x}\geq t}\abs{\tlj(x-y)-\tlj(x)}\, dx
&\leq \sum_{j\leq -m }\int_0^1\int_{\RR^n}\abs{\langle y,\nabla (\klj*\mathcal{F}^{-1}(\Psi_j))(x-\tau y) \rangle }\, dx\,d\tau \\
&\lesssim\sum_{j\leq -m } t \norm{\nabla (\klj*\mathcal{F}^{-1}(\Psi_j))}_1\\
&\lesssim\sum_{j\leq -m } t2^j \norm{\klj*\mathcal{F}^{-1}(\Psi_j)}_1
\lesssim_{\Psi}\sum_{j\leq -m }  \norm{\klj}t2^j
\leq l^{-\alpha}.
\end{align*}
The last two estimates give (\ref{c}), and the proof is complete.
\end{proof}

The following theorem now follows almost immediately using standard $L^p-$interpolation theory.
\begin{thm}
\label{lpforSl}
$T^l$ defines a bounded operator on $L^p$ for $p\in (1,\infty)$, with 
\[\norm{T^l}_{L^p}\lesssim_d \text{max }(p,(p-1)^{-1}Bl^{-\alpha+1}\]
for $l>0$. We also have
\[\norm{T^0}_{L^p}\lesssim_d \text{max }(p,(p-1)^{-1}B.\]
\end{thm}
\begin{proof}
The operator $T^l$ is bounded on $L^2$ (by Theorem \ref{l2bound}) and maps $L^1$ to $L^{(1,\infty)}$ (by Theorem \ref{weak thm}), with norm $Bl^{-\alpha+1}$ (norm $B$ for $l=0$). Interpolating between the $L^1$ and $L^2$ spaces then yields the required norm for the operator acting on $L^p$ with $p\in (1,2)$. Further, a duality argument yields the desired result for $p\in (2,\infty)$ as well. This concludes the proof of the theorem.
\end{proof}

\section{Boundedness on \texorpdfstring{$L^p$}{} for large \texorpdfstring{$p$}{}}
\label{largesec}
In this section, we investigate the properties of the truncated operator $\tlog f:= \sum_{j\leq n} \tlj*f$ (for $n\in \ZZ$ and $l\geq 0$), with an aim of establishing $L^p\rightarrow L^p$ bounds for the associated maximal operator $T^l_*f=\text{sup }\,|\tlog f|$. Let $M[f]$ denote the standard Hardy-Littlewood Maximal function. 
 
We wish to show that for $l\geq 0$, the operator $T^l_*$ is bounded on $L^p$ for $1<p<\infty$ with the corresponding operator norm $\lesssim_p Bl^{-\alpha+1}$ ($\lesssim_p B$ for $l=0$). We will achieve this by proving a Cotlar type inequality 
\[
T^l_* f(x)\lesssim_d \frac{1}{(1-\delta)^{1/q}}(M(\abs{T^lf}^q)(x))^{1/q}+2^\frac{ld}{q}l^{-\alpha+1}(1+\delta^{-1/q})B(m)(M(\abs{f}^q)(x))^{1/q}.
\]
for each $T^l_*$ and for a large enough exponent $q$ (here $0<\delta\leq 1/2$). Roughly speaking, a choice of $q=q_l \sim l$ will work. In particular, $q_l\rightarrow \infty$ as $l\rightarrow\infty$. Using this inequality, we can conclude that $T^l_*$ is bounded (with norm $\lesssim_pBl^{-\alpha+1})$ for all $p\in (p_l,\infty)$. 

The following lemma is the main step in establishing the Cotlar type inequality. The ideas used are similar to Lemma A.1 in [\cite{guo2019maximal}.
\begin{lem}\label{mainlemma}
Fix $\Tilde{x}\in \RR^d$, $n\in \ZZ$ and $q>1$. Let $g(y)=f(y)\mathds{1}_{B(\Tilde{x},2^{-n})}(y)$ and $h=f-g$. Then we have
\begin{enumerate}[label=(\roman*)]
    \item \label{l1} 
    $|\tlog g(\x)| \lesssim
                      \begin{cases}
                      B(\mq{f}(\x))^{1/q}\,&l=0.\\
                      0\, & l>0.
                      \end{cases}$
    \item \label{l2} $ 
                    |\tlog h(\x)-T^l h(\x)| \lesssim
                    \begin{cases}
                       0\, & l=0.\\
                       B2^{ld/q}l^{-\alpha+1}(\mq{f}(\x))^{1/q}\,&l>0.
                     \end{cases}$
                      
    \item \label{l3} For $\abs{w-\x}\leq 2^{-n-1}$, we have\\ 
    $|T^lh(\x)-T^lh(w)| \lesssim
    \begin{cases}
    B(\mq{f}(\x))^{1/q}\,&l=0.\\
    B2^{ld/q}l^{-\alpha+1}(\mq{f}(\x))^{1/q}\,&l>0.
    \end{cases}$
\end{enumerate}
\end{lem}
\begin{proof}
We may assume without loss of generality that $B=1$.
To prove \ref{l1}, we consider for $j\leq n$
\begin{align*}
    \abs{\tlj*g(\x)}
    &\leq \int_{\abs{\x-y}\leq 2^{-n}}\abs{\tlj(\x-y)g(y)}\,dy\\
    &\leq \norm{\maj(\x-y)}_{L_{q'}}\Bigg(\int_{\abs{\x-y}\leq 2^{-n}}\abs{g(y)}^q|\clj(\x-y)|^q\,dy\Bigg )^{1/q}.
\end{align*}
We note that $\clj(\x-y)$ is supported around $\x$ in either a dyadic annulus of radius $\sim 2^{-j+l}$ for $l>0$ or a disc of radius $2^{-j}$ for $l=0$. As $j\leq n$, the second term above is non-zero only when $l=0$. In this case, we estimate
\begin{align*}
    &\abs{T^{0,j}*g(\x)}\\ 
    &\lesssim 2^{jd/q}(\log{2})^{-\alpha} 2^{-nd/q}\Bigg(2^{nd}\int_{\abs{\x-y}\leq 2^{-n}}\abs{g(y)}^q\,dy\Bigg )^{1/q} \mathrm{(using\, Proposition\, \ref{new})}\\
    &\lesssim  2^{(j-n)d/q} (\mq{g}(\x))^{1/q}.
\end{align*}
Summing up in $j<n$, the assertion now follows as $\abs{g}\leq \abs{f}$.

For \ref{l2}, we observe that $\abs{\tlog h(\x)-T^l(\x)}\leq \sum_{j>n}\abs{\tlj*h(\x)}$. 
For $j>n$, we then have 
\begin{align*}
    \abs{\tlj*h(\x)}
    &\leq \int_{\abs{\x-y}\geq 2^{-n}}\abs{\tlj(\x-y)h(y)}\,dy\\
    &\leq \norm{\maj(\x-y)}_{L_{q'}}\Bigg(\int_{\abs{\x-y}\geq 2^{-n}}\abs{h(y)}^q|\clj((\x-y))|^q\,dy\Bigg )^{1/q}.
\end{align*}
Again, using the support property of $\clj$, we observe that the second term above is non-zero only when $l>0$ and $j<l+n$. For each $j\in (n,l+n)$, we estimate
\begin{align*}
&\abs{\tlj*h(\x)}\\
    &\leq 2^{jd/q}(\log{(2+2^l)})^{-\alpha}\Bigg(\int_{\abs{\x-y}\geq 2^{-n}}\abs{h(y)}^q|\clj(2^j(\x-y))|^q\,dy\Bigg )^{1/q} \mathrm{(using\, Proposition\, \ref{new})}\\
    &\lesssim 2^{jd/q} l^{-\alpha}\Bigg(\int_{\abs{\x-y}\sim 2^{-l+j}}\abs{h(y)}^q\,dy\Bigg )^{1/q} 
    \leq 2^{jd/q} l^{-\alpha}2^{(-j+l)d/q}\Bigg(2^{(-j+l)d}\int_{\abs{\x-y}\leq 2^{-l+j}}\abs{h(y)}^q\,dy\Bigg )^{1/q} \\
    &\leq l^{-\alpha}2^{ld/q}(\mq{h}(\x))^{1/q}.
\end{align*}
Summing up in $n<j<n+l$ and noting that $\abs{h}\leq \abs{f}$, we get
\[
\sum_{j>n}\abs{\tlj*h(\x)}\lesssim 2^{ld/q}l^{-\alpha+1}(\mq{f}(\x))^{1/q}.
\]
Now for \ref{l3}, we consider the terms $\tlj*h(\x)-\tlj*h(w)$ separately for $j\leq n$ and $j>n$. The sum $\sum_{j>n}\abs{\tlj*h(\x)}$ was already dealt with in \ref{l2} and like before, only matters for $l>0$. Since $\abs{w-\x}\leq 2^{-n-1}$ we have $\abs{w-y}\approx \abs{\x-y}$ for $\abs{\x-y}\geq 2^{-n}$ and for non-zero $l$, the previous calculation leads to
\[
\sum_{j>n}\abs{\tlj*h(w)}\lesssim
                    \begin{cases}
                       0\, & l=0.\\
                       B2^{ld/q}l^{-\alpha+1}(\mq{f}(\x))^{1/q}\,&l>0.
                     \end{cases}
\]
It remains to consider the terms for $j\leq n$. We write
\[
\tlj*h(\x)-\tlj*h(w)= \int_0^1\int_{\abs{\x-y}\geq 2^{-n}}\Big \langle \x-w, \nabla (\maj)(w+s(\x-w)-y)\Big \rangle h(y)dy ds.
\]
Since $\abs{w-\x}\leq 2^{n-1}$ we can replace $\abs{w+s(\x-w)-y}$ in the integrand with $\abs{\x-y}$. Also Proposition \ref{new} yields 
\[
\abs{\nabla (\maj)(\x-y)}\leq 2^j \abs{(\maj)(\x-y)}.
\]
Thus we have
\begin{align*}
&\abs{\tlj*h(\x)-\tlj*h(w)}\\
&\lesssim 2^j\abs{\x-w}\int_{\abs{\x-y}\geq 2^{-n}} \abs{(\maj)(\x-y)h(y)}dy\\
&\leq 2^{j-n-1}\norm{(\maj)(\x-y)}_{L_{q'}}\Bigg(\int_{\abs{\x-y}\geq 2^{-n}}\abs{h(y)}^q|\clj((\x-y))|^q\,dy\Bigg )^{1/q}.
\end{align*}
Proposition \ref{new} now gives us
\begin{align*}
\abs{\tlj*h(\x)-\tlj*h(w)}
&\lesssim  2^{j-n-1}2^{jd/q}(\log{(2+2^l)})^{-\alpha}\Bigg(\int_{\abs{\x-y}\sim 2^{-j+l}}\abs{h(y)}^q\,dy\Bigg )^{1/q} \\
&\lesssim  2^{j-n-1}2^{jd/q}(\log{(2+2^l)})^{-\alpha}2^{(-j+l)d/q}(\mq{h}(\x))^{1/q} \\
&\lesssim \begin{cases}
    2^{ld/q} 2^{j-n-1}(\mq{f}(\x))^{1/q}\, & l=0.\\
    2^{ld/q} 2^{j-n-1}l^{-\alpha}(\mq{f}(\x))^{1/q}\,&l>0.
    \end{cases}
\end{align*}
Summing in $j\leq n$ leads to
\[
\sum_{j\leq n}\abs{\tlj*h(\x)-\tlj*h(w)}\lesssim 
\begin{cases}
    2^{ld/q} (\mq{f}(\x))^{1/q}\, & l=0.\\
    2^{ld/q} l^{-\alpha}(\mq{f}(\x))^{1/q}\,&l>0.
\end{cases}
\]
\end{proof}

We can now prove the main result of this section. 
\begin{prop}\label{cotlarlemma}
Let $\alpha>3$, $q>1$ and $B(m)$ be as in (\ref{assump}). Let $f$ be a Schwartz function. Then for almost every $x$ and for $0<\delta \leq 1/2$, we have
\[
T^l_* f(x)\lesssim \frac{1}{(1-\delta)^{1/q}}(M(\abs{T^lf}^q)(x))^{1/q}+C_dA_l(1+\delta^{-1/q})(M(\abs{f}^q)(x))^{1/q}
\]
where
\[
A_l=
\begin{cases}
                      B\,&l=0\\
                      B2^{ld/q} l^{-\alpha+1}\, & l>0.
                      \end{cases}
\]
\end{prop}
\begin{proof}
The proof is essentially the same as that of an analogous result in [\cite{guo2019maximal}, which is in turn a modification of the argument for the standard Cotlar inequality regarding the truncation of singular integrals (see [\cite{stein1993harmonic}, sec 1.7).

Fix $\x\in \RR^d$ and $n\in \ZZ$ and define $g$, $h$ and $q$ as in the previous lemma. For $w$ (to be chosen later) with $\abs{w-\x}\leq 2^{-n-1}$ we can write
\begin{align}
\tlog f(\x)
&= \tlog g(\x)+(\tlog-T^l)h(\x)+T^lh(\x)\nonumber\\
&= \tlog g(\x)+(\tlog -T^l)h(\x)+T^lh(\x)-T^lh(w)+T^lf(w)-T^lg(w). 
\label{terms}
\end{align}
By Lemma \ref{mainlemma}, we have
\[
\abs{\tlog g(\x)}+\abs{(\tlog-T^l)h(\x)}+\abs{T^lh(\x)-T^lh(w)}\lesssim A_l(\mq{f}(\x))^{1/q}.
\]
All that remains in this case is to consider the term $T^lf(w)-T^lg(w)$ for $w$ in a substantial subset of $B(\x,2^{-n-1})$.
By Theorem \ref{weak thm} we have that for all $f\in L^q(\RR^d)$ and all $\lambda>0$
\[
\textrm{meas}\{x\colon \abs{T^lf(x)}>\lambda\}\leq A_l^q\lambda^{-q}\norm{f}_q^q.
\]
Now let $\delta\in (0,1/2)$ and consider the set 
\[
\Omega_n(\x,\delta)=\{w\colon \abs{w-\x}<2^{-n-1}, \, \abs{T^lg(w)}>2^{d/q}\delta^{-1/q}A_l(\mq{f}(\x))^{1/q}\}.
\]
In (\ref{terms}) we can estimate the term $\abs{T^lg(w)}$ by $A_l2^{d/q}\delta^{-1/q}(\mq{f}(\x))^{1/q}$ whenever $w\in B(\x, 2^{-n-1})\setminus \Omega_n(\x,\delta)$. Hence we obtain
\begin{equation}
\abs{\tlog f(\x)}\lesssim \underset{w\in B(\x, 2^{-n-1}\setminus \Omega_n(\x,\delta)}{\textrm{inf}} \abs{T^lf(w)}+C_dA_l(1+\delta^{-1/q})(\mq{f}(\x))^{1/q}.
\end{equation}
By the weak type inequality for $T^l$ we have
\begin{align*}
\textrm{meas}(\Omega_n(\x,\delta))
&\leq \frac{A_l^q\norm{g}_q^q}{2^d\delta^{-1}A_l^q\mq{f}(\x)}=\frac{\delta}{2^d\mq{f}(\x)}\int_{\abs{x-y}\leq 2^{-n}}\abs{f(y)}^q\, dy\\
&\leq \delta 2^{-d}\textrm{meas}(B(\x,2^{-n})=\delta \textrm{meas}(B(\x,2^{-n-1}).
\end{align*}
Hence $\textrm{meas}(B(\x, 2^{-n-1})\setminus \Omega_n(\x,\delta))\geq (1-\delta)\textrm{meas}(B(\x,2^{-n-1})$ and thus
\begin{align*}
\underset{w\in B(\x, 2^{-n-1}\setminus \Omega_n(\x,\delta))}{\textrm{inf}} \abs{T^lf(w)}
&\leq \Bigg ( \frac{1}{\textrm{meas}(B(\x, 2^{-n-1}\setminus \Omega_n(\x,\delta))}\int_{B(\x, 2^{-n-1})}\abs{T^lf(w)}^q\Bigg )^{1/q}\\
&\leq \Bigg ( \frac{1}{(1-\delta)\textrm{meas}(B(\x,2^{-n-1}))}\int_{B(\x, 2^{-n-1})}\abs{T^lf(w)}^q\Bigg )^{1/q}.
\end{align*}
We obtain 
\[
\abs{\tlog f(\x)}\lesssim \frac{1}{(1-\delta)^{1/q}}(M(\abs{T^lf}^q)(\x))^{1/q}+A_l(1+\delta^{-1/q})(M(\abs{f}^q)(\x))^{1/q}
\]
uniformly in $n$, which implies Proposition \ref{cotlarlemma}.
\end{proof}

The above proposition, in conjunction with Theorem \ref{lpforSl} immediately leads to the following:
\begin{thm}\label{large thm}
$T_*^l$ is bounded on $L^p$ for $p\in (l+2,\infty)$. 
\end{thm}
\begin{proof}
Fix $\delta=1/2$ and $q=l+2$. We note that both $f\rightarrow T^lf$ and $f\rightarrow (\mq{f})^{1/q}$ are bounded operators on $L^p$ for $p\in (l+2,\infty)$, with operator norms bounded by $pBl^{-\alpha+1}$ ($pB$ for $l=0$) and $p/(p-l-2)$ respectively (upto multiplication by a dimensional constant). Hence, by Proposition \ref{cotlarlemma}, we obtain
\[
\norm{T^l_*}_{L^p}\lesssim_d 
\begin{cases}
B, &l=0\\
l^{-\alpha+1}\frac{p^2}{p-l-2}B, &l>0.
\end{cases}
\]
\end{proof}

\section{Weak (1,1) Boundedness of the Maximal Operator }
\label{weakmax}
In this section, we will prove that each of the pieces $T^l_*$ is of weak type $(1,1)$, with the respective norm $\lesssim_d Bl^{-\alpha+1}$ ($\lesssim_d B$ for $l=0)$. Combining this result with Theorem \ref{large thm} and interpolating, we will obtain bounds on the operator norm of $T^l_*$ on $L^p$ for all $p\in (1,\infty)$, with a decay of $l^{-\alpha+2}$. This will allow us to achieve our final goal of summing up the pieces together to get bounds on the operator $S^*$. 

The proof we give here is essentially the same as the one for Theorem \ref{weak thm}, except for one notable difference. In order to obtain a weak bound for the "good" function in the Calder\'on-Zygmund decomposition, we use the bounds on $L^{p_l}$ (as given by Theorem \ref{large thm}, with $p_l=4l)$ in place of those on $L^2$, noting that the operator norm in the former case is $\lesssim Bl^{-\alpha+2}$ (for $l>0$). The upshot is that the power of $l$ in the weak (1,1) norm of $T^l_*$ goes up by $1$. Instead of repeating the entire argument, we sketch an outline.

\begin{thm}
\label{weak2}
$T_*^l$ is weak (1,1) bounded with
\[
\norm{T_*^l}_{L^1\rightarrow L^{1,\infty}}\lesssim
\begin{cases}
B, &l=0,\\
Bl^{-\alpha+2},&l>0. 
\end{cases}
\]
\end{thm}
\begin{proof}
Let $n\in \ZZ$ and $f\in L^1(\RR^d)$. Fix $\sigma>0$. We sketch the proof for $l>0$ (the one for $l=0$ proceeds in almost the same way). Also, we might assume $B=1$. As before, we make a Calder\'on-Zygmund decomposition of $f$ at the level $l^{\alpha-1}\sigma$. For the "good" function $f_1$, we use the bound $\norm{\tlog}_{L^{P_l}}\lesssim Bl^{-\alpha+2}$ on $L^{p_l}$ with $p_l=4l$ for $l>0$ and $\norm{T^0_n}_{L^{3}}\lesssim B$ on $L^{3}$. We obtain for $l>0$
\[
\norm{\tlog f_1}_{L^{1,\infty}}\lesssim Bl^{-\alpha+2}\norm{f}_{L^1}
\]
(and the corresponding result for $l=0$). The argument for the "bad" part $f_0$ proceeds the exact way as in proof of Theorem \ref{large thm}, with the index $n$ playing no real role, and we get for $l>0$
\[
\norm{\tlog f_0}_{L^{1,\infty}}\lesssim Bl^{-\alpha+1}\norm{f}_{L^1}
\]
(and the corresponding result for $l=0$). The result then follows by combining the two estimates followed by taking a supremum over $n$.
\end{proof}
Theorems \ref{large thm} and \ref{weak thm} together via $L^p$ interpolation lead to 
\begin{thm}\label{large thm2}
$T_*^l$ is bounded on $L^p$ for $p\in (1,\infty)$, with 
\[
\norm{T^l_*}_{L^p}\lesssim_d 
\begin{cases}
B\text{ max }(p,(p-1)^{-1}), &l=0\, , p\in(1,\infty)\\
B\text{ max }(p,(p-1)^{-1})l^{-\alpha+2}, &l>0\, , p\in(1,4l)\\
Bpl^{-\alpha+1}, &l>0\, , p\in[4l,\infty).
\end{cases}
\]
\end{thm}
\begin{proof}
The result for $l=0$ is a clear outcome of $L^p$ interpolation and the estimates proved earlier. For $l>0$ and $p\in [4l,\infty)$, it follows easily from Theorem \ref{large thm} and the observation that $p^2/(p-l-2)\lesssim p$ for $p\geq 4l$.
For $l>0$ and $p\in (1,4l)$, it is an outcome of interpolating between the weak (1,1) estimate in Theorem \ref{weak2} and the one contained in Theorem \ref{large thm} for $p=4l$, again making the observation that $p^2/(p-l-2)\lesssim p=4l$.
\end{proof}

Finally, we give the proof of Theorem \ref{main thm}.
\begin{proof}
It is enough to prove the theorem for $S_n$ for a fixed $n\in \ZZ$, as the result then follows by taking the supremum over $n\in \ZZ$. For $p\in (1,\infty)$ and a Schwartz function $f$, we have
\[\norm{S_nf}_{L^p}\leq \sum_{l\geq 0}\norm{\slog f}_{L^p}\leq \sum_{l\geq 0}\norm{\tlog f}_{L^p}+\sum_{l\geq 0}\norm{\sum_{j\leq n}\klj *\mathcal{F}^{-1}(1-\Psi_j)*f}_{L^p}.\]
Using Theorem \ref{large thm2} for the first sum and Theorem \ref{errorthm} for the second one (and noting that summing up for $j\leq n$ instead of $j\in \ZZ$ does not affect the proof), we get
\[\norm{S_nf}_{L^p}\lesssim B \text{ max }(p,(p-1)^{-1})\sum_{l\geq 1}(l^{-\alpha+2}+2^{-l})\lesssim_p B.\]
For weak (1,1) boundedness, we argue in a similar way, only using Theorem \ref{weak2} this time in place of Theorem \ref{large thm2}.
\end{proof}

\nocite{*}
\bibliography{main}

\end{document}